\documentclass[12pt]{amsart}
\usepackage{amsmath,amsthm,latexsym,amscd,amsbsy,amssymb,url}
\usepackage[all]{xypic}
 \relax

\chardef\bslash=`\\ 

\makeatletter
\def\verbatim{\interlinepenalty\@M \@verbatim
  \leftskip\@totalleftmargin\advance\leftskip2pc
  \frenchspacing\@vobeyspaces \@xverbatim}
\makeatother
\newtheorem{thm}{Theorem}[section]
\newtheorem{cor}[thm]{Corollary}
\newtheorem{lem}[thm]{Lemma}

\begin{document}


\title
{Stationary Sets in Topological and Paratopological Groups}
\author{Raushan ~Z.~Buzyakova}
\address{}
\email{Raushan\_Buzyakova@yahoo.com}
\author{Cetin Vural}
\address{Department of Mathematics, Faculty of Arts and Sciences, Gazi University, 06500 Ankara, Turkey}
\email{cvural@gazi.edu.tr}
\keywords{topological group, paratopological group, paracompact,   ordinal, stationary subset of a regular uncountable cardinal}
\subjclass{54H11, 22A05, 54D20}


\begin{abstract}{
We show that if a topological or paratopological group $G$ contains a stationary subset of
some regular uncountable cardinal, then $G$ contains a subspace which is not collectionwise normal.
This statement implies that if a monotonically normal space (in particular, any generalized ordered space) is a paratopological group then the space is hereditarily paracompact. }
\end{abstract}

\maketitle
\markboth{Raushan Z. Buzyakova and Cetin Vural}{Stationary Sets in Topological and Paratopological Groups}
{ }

\section{Introduction}\label{S:intro}

\par\bigskip
It is proved in \cite{Buz} that if under certain conditions a group $G$ contains a subset homeomorphic to an uncountable ordinal then $G$ is not hereditarily normal. This statement  suggests that a similar statement should be true
if one replaces "an uncountable ordinal" by "a stationary subset of an uncountable regular cardinal". The purpose of this paper
is just this task. The main result of the paper  states that if a topological or paratopological group contains a stationary subset of an uncountable regular cardinal,
then the group is not hereditarily collectionwise normal (or, more precisely, contains a subspace which is either not normal
or not collectionwise Hausdorff).  
Using our main result and  
the  Balogh-Rudin generalization \cite{BalRud}
of the Engelking-Lutzer theorem \cite{Lut}, we conclude that every monotonically
normal topological (or paratopological) group is hereditarily paracompact. This significantly improves an earlier
result in \cite{Vur}, namely, by considerably relaxing the hypothesis
and significantly strengthening the conclusion.

In notation and terminology we will follow \cite{Eng}. All spaces under consideration are Tychonov.
If $X$ is a generalized ordered space and 
$Y$ is its subspace then by $[a,b]_Y$ we denote the trace of the segment $[a,b]\subset X$
in $Y$. The same concerns other types of intervals. If $f$ is a continuous function
from $X$ to $Y$, by $\tilde f$ we denote its continuous extension to the \v Cech-Stone compactifications
$\beta X$ and $\beta Y$. As usual, $G$ is {\it a paratopological group} if $G$ is a group, a topological space, and the group binary operation is continuous with respect to the topology of $G$. Basic facts about ordinals and their stationary subsets
can be found, in particular, in \cite{Kun} or any other introductory text on set theory.
Ordinals and their subsets, when considered as topological spaces, are endowed with the topology of linear order and
the subspace topology, respectively. 
\par\bigskip

\section{Result}\label{S:Result}

\par\bigskip
We start with the statement which is a corollary to the classical argument of Katetov \cite{Kat}
(or see Exercise 2.7.15(a) in \cite{Eng}).
\par\bigskip\noindent
\begin{lem}\label{lem:KatetovProduct} ({\bf Follows from \cite{Kat}})
Suppose that $S$ is a stationary subset of a regular uncountable cardinal $\kappa$, $\lambda<\kappa$,
$\lambda\in L\subset [0,\lambda]$, $\lambda$ is a limit point of $L$, and 
$$
K=\{\langle x,y\rangle\in L\times S : (x<\lambda)\ or \ (x=\lambda\ and \ y\ is\ isolated\ in\ S)\}.
$$
Then $A= \{\langle \lambda, y\rangle\in K\}$ and $B= \{\langle x,y\rangle\in K: y\ is\ limit\ in \ S\}$
are closed and disjoint subsets of $K$ that cannot be separated by disjoint open neighborhoods in $K$. 
\end{lem}
\begin{proof}({\it Natural modification of the argument in \cite{Kat}})
It is clear that $A$ and $B$ are closed and disjoint in $K$.  
Let $U$ be any open set in $K$ that contains $B$. We need to show that $A$ meets the closure of $U$ in $K$.
By stationarity of $\{\alpha\in S: \alpha\ is\ a\ limit\ point\ for\ S\}$, for every $x\in L\setminus \{\lambda\}$ we can find $y_x<\kappa$ such that
$\{x\}\times [y_x, \kappa)_S\subset U$. Since $|L|<\kappa$,
we have $\bar y =\sup \{y_x:x\in L\setminus \{\lambda\}\}<\kappa$. Thus, $(L\setminus \{\lambda\})\times [\bar y,\kappa)_S$
is a subset of $U$, meaning that the closure of $U$ meets $A$.
\end{proof}

\par\bigskip\noindent
\begin{lem}\label{lem:OneToOneOnKatetovProduct}
Let $G$ be a paratopological group with binary operation $\star$ and let $T\subset G$
be a stationary subset of a regular uncountable cardinal $\kappa$. Then 
it is possible to choose a closed unbounded subset $S$ of the set $T$,
an ordinal $\lambda$ in $T$, and a set $L\subset [0,\lambda]_T$ that have the following properties:
\begin{enumerate}
	\item $\lambda\in L$,
	\item $\lambda$ is a limit point of $L$, and
	\item the operation $\star$ is injective on $L\times S$.
\end{enumerate}
\end{lem}
\begin{proof} 
We will start with the following claim.
\par\bigskip\noindent
{\it Claim.} {\it
For any $\alpha\in T$ there exists $\lambda_\alpha\in T$ such that $\alpha\star y\not = x\star z$
whenever $x,y,z\in [\lambda_\alpha,\kappa)_T$.
}

\par\smallskip\noindent
To prove the claim we assume that no such $\lambda_\alpha$ exists. Then for each $\gamma\in T$
we can find $x_\gamma,y_\gamma,z_\gamma\in T$ with the following properties:
\begin{description}
	\item[{\rm P1}] $\alpha\star y_\gamma = x_\gamma\star z_\gamma$;
	\item[{\rm P2}] $x_\gamma,y_\gamma,z_\gamma\in (\gamma,\kappa)_T$; 
	\item[{\rm P3}] $x_\gamma,y_\gamma,z_\gamma\in (\max\{x_\beta,y_\beta,z_\beta\},\kappa)_T$ if $\gamma>\beta$.
\end{description}
By stationarity of $T$ we can find $\bar \gamma\in T$ a limit point for $T$ such that
\begin{description}
	\item[{\rm P4}] $\bar\gamma = \sup \{x_\gamma:\gamma<\bar\gamma\} 
	=\sup \{y_\gamma:\gamma<\bar\gamma\}=\sup \{z_\gamma:\gamma<\bar\gamma\}$
\end{description}
By continuity of $\star$ and P1, we have $\alpha\star\bar\gamma=\bar\gamma\star\bar\gamma$. This equality 
implies that 
$\alpha=\bar\gamma$, which contradicts P2 and P3 and proves the claim.

\par\bigskip\noindent
By claim we can find an unbounded  subset $P\subset T$ with the following property:
\begin{description}
	\item[(*)] $\lambda_\alpha <\beta$ whenever $\alpha,\beta\in P$ and $\alpha<\beta$.
\end{description}
Now, by stationarity of $T$, we can find a point $\lambda\in T$ which is a limit point of the set $P$.
Put $S = [\lambda,\kappa)_T$ and $L=\{\alpha\in P:\alpha\leq \lambda\}\cup\{\lambda\}$. 
By (*) and the property of $\lambda_\alpha$'s in the claim conclusion, $S$ and $L$ are as desired.
\end{proof}

\par\bigskip\noindent
\begin{lem}\label{onetooneimageofproduct}
Suppose that $S$ is a stationary subset of an uncountable regular cardinal $\kappa$,
$\lambda<\kappa$, $\lambda\in L\subset [0,\lambda]$, $\lambda$  is a limit point for $L$, 
and $f$  is a continuous bijection
of $L\times S$ onto its image $Z$ such that 
$f|_{\{\lambda\}\times S}$ is a homeomorphism.
Then  $Z$ contains a subspace which is not normal or not collectionwise Hausdorff.
\end{lem}
\begin{proof}
We may assume that $Z$ is normal. 
To simplify our formulas put
$$
V_\lambda = \{\lambda\}\times S\ and\ V_{\lambda}^\alpha = \{\lambda\}\times (\alpha, \kappa)_S.
$$
\par\bigskip\noindent
{\it Claim 1.} {\it
If $\langle x_1,y_1\rangle,\ \langle x_2,y_2\rangle\in L\times S$ and
$f(x_1,y_1),f(x_2,y_2)\in Cl_Z(f(V_\lambda^\alpha))$ for all $\alpha<\kappa$,
then $\langle x_1,y_1\rangle = \langle x_2,y_2\rangle$.
}
\par\smallskip\noindent
To prove the claim, we first observe that any Hausdorff compactification of $V_\lambda$ has only one
point of complete accumulation of $V_\lambda$. This follows from the fact that $\kappa$ is regular and $S$ is stationary in
$\kappa$. Denote this point in $Cl_{\beta (L\times S)} (V_\lambda)$ by $p$. Then by the claim assumption we
have $\tilde f(p)=\tilde f (x_1,y_1)=\tilde f(x_2,y_2)$, which implies that $f(x_1,y_1)=f(x_2,y_2)$.
Since $f$ is injective, we arrive at the conclusion of the claim.
\par\bigskip\noindent
Note that if $f(x,y)$ is not a complete accumulation point of $f(V_\lambda)$ but belongs
to the closure of $f(V_\lambda)$ in $Z$, then $f(x,y)$ belongs to the closure
of $f(\{\lambda\}\times [0,\alpha)_S)$ in $Z$ for some $\alpha<\kappa$.
\par\bigskip\noindent
{\it Claim 2.} {\it
There exists $\alpha<\kappa$ such that $f(V_\lambda^\alpha)$ 
is closed in $Z_\alpha = f(L\times (\alpha, \kappa)_S)$.
}
\par\smallskip\noindent
To prove the claim, we assume the contrary. Then, by Claim 1,  for every $\alpha<\kappa$ we can find a pair
$\langle x_\alpha,y_\alpha\rangle\in [L\times (\alpha,\kappa)_S]\setminus V_\lambda^\alpha$ and 
an element $z_\alpha\in (\alpha,\kappa)$
with the following properties:
\begin{description}
	\item[\rm 1] $f(x_\alpha,y_\alpha)\in Cl_{Z_\alpha}(f(\{\lambda\}\times (\alpha,z_\alpha)_S))$,
	\item[\rm 2] $y_\alpha,z_\alpha >\max\{y_\beta,z_\beta,\beta\}$ for ever $\alpha>\beta$.
\end{description}
Since $cf(\kappa)=\kappa>\lambda$, we can also assume that
\begin{description}
	\item[\rm 3] $x_\alpha=x_\beta=x$ for all $\alpha,\beta<\kappa$.
\end{description}
Also, by our choice:
\begin{description}
	\item[\rm 4] $x\not = \lambda$.
\end{description}
By (2) and stationarity of $S$, we can find an element $y\in S$ that has the following property:
\begin{description}
	\item[\rm 5] $y=\sup\{y_\alpha:\alpha<y\}=\sup\{z_\alpha:\alpha<y\}$.
\end{description}
Applying (1),(3),(5) and continuity of $f$, we conclude that $f(\lambda, y)= f(x,y)$. This equality and property (4) contradict
the fact that $f$ is injective. Claim 2 is proved.
\par\bigskip\noindent
By virtue of Claim 2, we may assume that $f(V_\lambda)$ is closed in $Z$.
\par\bigskip\noindent
{\it Claim 3.} {\it
For any $\alpha<\kappa$ there exist $a_\alpha<\lambda$ and $b_\alpha<\kappa$ such that
the closure of $f([a_\alpha, \lambda]_L\times [b_\alpha, \kappa)_S)$ in $Z$ misses the closure of
$f([a_\alpha,\lambda ]_L\times \{\alpha\})$ in $Z$.
}
\par\smallskip\noindent
Since $f(V_\lambda )$ is closed in $Z$, the set $V_\lambda^\alpha$ is closed in $V_\lambda$, and $f|_{V_\lambda}$ is a homeomorphism, we conclude
the set $f(V_\lambda^{\alpha})$
is closed in $Z$. Since $f$ is one-to-one, 
$f(\lambda,\alpha )\not \in f(V_\lambda^{\alpha})$. By regularity of $Z$,
we can find open 
sets $U$ and $W$ in $Z$ that contain 
$f(V_\lambda^{\alpha})$
and $f(\lambda,\alpha )$, respectively, such that $Cl_Z(U)\cap Cl_Z(W)=\emptyset$. 
Since $S$ is stationary in $\kappa$ and $\lambda<cf(\kappa)$, we can find an element
$b_\alpha<\kappa$ such that $\{\lambda\}\times [b_\alpha, \kappa)_S$ has a rectangular neighborhood
$[a'_\alpha, \lambda]_L\times [b_\alpha,\kappa)_S$ contained in $f^{-1}(U)$.
Since $\lambda$ is a limit point of $L$, we can find $a_\alpha''<\lambda$ such that
$[a_\alpha'',\lambda]\times \{\alpha\}\subset f^{-1}(W)$.
Clearly, $a_\alpha=\max\{a_\alpha',a_\alpha''\}$ and $b_\alpha$ are as desired. The claim is proved.

\par\bigskip\noindent
For the remaining argument put
$$
P_{\alpha}=L\times [0,\alpha]_S {\rm\ and\ as\ above}\ \ V_{\lambda}^\alpha = \{\lambda\}\times (\alpha, \kappa)_S.
$$
The rest of the proof is handling of the following two cases.
\par\bigskip\noindent
\underline {\it Case 1.} The assumption for this case is: {\it
"There exists $\alpha$ such that $Cl_Z(f(P_{\alpha}))$ meets
$f(V_\lambda^{\gamma})$ for every $\gamma>\alpha$."
}
To handle this case, for every 
$\gamma>\alpha$, fix $\langle \lambda, y_\gamma\rangle\in V_\lambda^{\gamma}$
such that $f(\lambda, y_\gamma)$ belongs to $Cl_Z(f(P_{\alpha}))$.
Thus, there exists a $\kappa$-sized $S'\subset S$ which is discrete in itself and
$f(\{\lambda\}\times S')$ is in the closure
of $f(P_{\alpha})$ in $Z$. Since $f(V_\lambda)$ is closed in $Z$ and $f|_{V_\lambda}$ is a homeomorphism, 
we conclude that the set $f(\{\lambda\}\times S')$ is of cardinality $\kappa$
and is closed and discrete in $f(\{\lambda\}\times S')\cup f(P_\alpha)$. Since the closure
of $f(P_{\alpha})$ in $Z$ has density at most $|\lambda\times \alpha |<\kappa$, we conclude
that $Cl_Z(f(P_{\alpha}))$ is not hereditarily collectionwise Hausdorff, which completes Case 1.

\par\bigskip\noindent
\underline {\it Case 2.} The assumption for this case is: {\it "Case 1 does not take place."} 
To handle this case  we will find $K,A$, and $B$ as in Lemma \ref{lem:KatetovProduct}.
Since Case 1 does not take place we can find a closed unbounded $S'$ in $S$ such that
\begin{description}
	\item[{\rm P1}] $Cl_Z(f(P_{\alpha}))\cap f(V_\lambda^{\beta}) =\emptyset$ for any $\alpha,\beta\in S'$, where $\alpha<\beta$
\end{description}
For each $\alpha\in S'$, fix $a_\alpha$ and $b_\alpha$ as in Claim 2. There exists $S''$ closed and unbounded in $S'$ such that 
\begin{description}
	\item[{\rm P2}] $\beta>b_\alpha$ for any $\alpha,\beta\in S''$ with $\alpha<\beta$.
\end{description}
In other words,
\begin{description}
	\item[{\rm P3}] $Cl_Z(f([a_\alpha,\lambda]_L\times \{\alpha\}))$  misses $Cl_Z(f([a_\alpha,\lambda]_L\times [\alpha+1,\kappa)_{S''})$.
\end{description}
By Claim 3 there exist a $\kappa$-sized subset $S'''\subset S''$ and an ordinal $a<\lambda$ such that
$a_\alpha=a$ for every $\alpha\in S'''$. Let $\hat{S} =Cl_S(S''')$.
Put $Y= [a,\lambda]_L\times \hat{S}$. By P1 and P3 we have
\begin{description}
	\item[{\rm P4}] $f(H_\alpha\cap Y)$ is clopen in $f(Y)$ for any $\alpha\in \hat S$.
\end{description}
We are ready to define $K, A$, and $B$ as  in Lemma \ref{lem:KatetovProduct}. Put $K=([a,\lambda)_L\times \hat S)\cup \{\langle\lambda, y\rangle:y\ is\ isolated\ in\ \hat S\}$.
We have  $A= \{\langle\lambda, y\rangle\in K\}$ and $B= \{\langle x,y\rangle\in K: y\ is\ limit\ in \ \hat S\}$
are closed and disjoint subsets of $K$ that cannot be functionally separated 
as  in Lemma \ref{lem:KatetovProduct}. By P4, the set $Cl_{f(K)}(f(B))$ misses $f(A)$. 
Since $f(V_\lambda)$ is closed in $Z$ and $f$ is injective, we conclude that $f(A)$ is closed in $f(K)$.
Since $A$ and $B$ cannot be functionally separated in $K$,
we conclude that  $Cl_{f(K)}(f(B))$ and $f(A)$ cannot be functionally separated in $f(K)$ either.
Since the closures of the images of $A$ and $B$ in $f(K)$ are disjoint we conclude that
$f(K)$ is not normal, which completes the second case and proves our statement.
\end{proof}

\par\bigskip
Now we are ready to derive our main results.
\par\bigskip\noindent
\begin{thm}\label{thm:main}
Let $G$ be a paratopological group. If $G$ contains a stationary subspace of an uncountable
regular cardinal, then the following are true:
\begin{enumerate}
	\item $G$ is not hereditarily normal or not hereditarily collectionwise Hausdorff;
	\item $G$ is not hereditarily collectionwise normal.
\end{enumerate}
\end{thm}
\begin{proof}
Clearly (2) follows from (1). Let us show that (1) holds.
Let $S$, $L$, and $\lambda$ be as in the conclusion of Lemma \ref{lem:OneToOneOnKatetovProduct}.
Since multiplication by an element of $G$ is an automorphism, we conclude that
$\star|_{\{\lambda\}\times S}$ is a homeomorphism.
 Therefore, $f=\star$, $S$, $L$, and $\lambda$ satisfy the hypothesis of 
Lemma \ref{onetooneimageofproduct}. 
 The conclusion of
Lemma \ref{onetooneimageofproduct} completes the proof.
\end{proof}

\par\bigskip
For our further discussion we would like to state the version of Jones result \cite{Jon}:
{\it Suppose that $X$ has density $\kappa$ and has a closed discrete subspace of cardinality $\tau$. If $2^\kappa<2^\tau$
then $X$is not normal.}
Observe that in Lemma \ref{onetooneimageofproduct}, the first case that leads to finding a subspace
which is not collectionwise Hausdorff and has a closed discrete subspace of cardinality greater than
the density.
Therefore if one assumes the Generalized Continuum Hypothesis and uses the mentioned Jones Lemma, then
both cases of the proof of  
Lemma \ref{onetooneimageofproduct}  lead to a subspace which is not normal. We can state this observation as follows.
\par\bigskip\noindent
\begin{thm}\label{thm:mainunderGCH} Assume the Generalized Continuum Hypothesis.
Let $G$ be a paratopological group. If $G$ contains a stationary subspace of an uncountable
regular cardinal, then $G$ is not hereditarily normal.
\end{thm}

\par\bigskip\noindent
It is of course natural to wonder if
 a paratopological group containing a stationary subset of an uncountable
regular cardinal is not hereditarily normal without any set-theoretic assumptions.

\par\bigskip

For our next corollary we need the fact that every monotonically normal space is hereditarily collectionwise
normal \cite[Theorem 3.1]{HLZ} and hereditarily monotonically normal (which follows from the definition).
\par\bigskip\noindent
\begin{thm}\label{thm:maincorollary} 
Let $G$ be a paratopological group. If $G$ is monotonically normal, then $G$ is hereditarily
paracompact.
\end{thm}
\begin{proof}
The Balogh-Rudin generalization \cite{BalRud} of the Engelking-Lutzer theorem states that if a monotonically
normal space is not paracompact then it contains a stationary subset of an uncountable
regular cardinal as a closed subspace. 
Now assume $G$ is not hereditarily paracompact.
Since monotone normality property is inherited by all subspaces
it follows from the Balogh-Rudin theorem that $G$ contains a stationary subset of an uncountable regular cardinal.
By (2) of  Theorem \ref{thm:main}, $G$ is not hereditarily collectionwise normal. However monotone normality implies
that every subspace of $G$ is collectionwise normal. This contradiction completes the proof.
\end{proof}

\par\bigskip
Since every subspace of a linearly ordered space is monotonically normal we have the following statement. 

\par\bigskip\noindent
\begin{cor}\label{cor:GO}
Let $G$ be a paratopological group. If $G$ is a subspace of a linearly ordered space then $G$ is hereditarily
paracompact.
\end{cor}

\par\bigskip\noindent
Corollary \ref{cor:GO} significantly improves an earlier result in \cite{Vur} in which the hypothesis has
one more additional condition, namely, that the group binary operation preserves the order, and the conclusion is paracompactness
of the entire space only.

\par\bigskip\noindent
{\bf Acknowledgment.} {\it
The authors would like to thank the referee for many helpful remarks, corrections, and suggestions.
}

\par
\end{document}